\documentclass[a4paper]{amsart}
\usepackage[foot]{amsaddr}
\usepackage[utf8]{inputenc}
\usepackage[nocompress]{cite}
\usepackage{graphicx}
\usepackage{subcaption}
\usepackage{epsf, tikz, tkz-graph}
\usepackage{amssymb,amsmath,eufrak,amsthm,amscd,mathtools}
\usepackage{graphics}
\usepackage{textcomp}
\usepackage{comment}
\usepackage{enumitem}
\usepackage{verbatim}
\usepackage{csquotes, nicefrac}
\usepackage{capt-of,caption}
\usepackage{graphicx,newfloat,float}
\usepackage{url}
\usepackage{breakurl}

\newtheorem{theorem}{Theorem}[section]
\newtheorem{lemma}[theorem]{Lemma}
\newtheorem{prop}[theorem]{Proposition}

\newtheorem*{statement*}{Statement}
\newtheorem*{theorem*}{Theorem}
\newtheorem*{lemma*}{Lemma}
\newtheorem*{fact*}{Fact}

\theoremstyle{definition}

\newtheorem*{definition*}{Definition}

\newtheorem*{example*}{Example}

\newtheorem*{exercise*}{Exercise}

\newtheorem*{proposition*}{Proposition}
\newtheorem{corollary}[theorem]{Corollary}
\newtheorem*{corollary*}{Corollary}

\newtheorem*{claim*}{Claim}

\newtheorem*{test*}{Test}

\theoremstyle{remark}

\newtheorem*{remark*}{Remark}

\makeatletter

\newcommand{\R}{\mathbb{R}}
\newcommand{\Z}{\mathbb{Z}}
\newcommand{\C}{\mathbb{C}}

\newcommand{\Hyp}{\mathbb{H}}
\newcommand{\e}{\varepsilon}
\newcommand{\supp}{\text{supp}}

\newcommand{\PSL}{\text{PSL}_2(\mathbb{R})}

\newcommand{\l@abcd}[2]{\hbox to\textwidth{#1\dotfill #2}}

\title{On the support of a non-autocorrelated function on a hyperbolic surface}
\author[K. Golubev]{Konstantin Golubev}
\email{golubevk@ethz.ch}
\address{D-MATH, ETH Zurich, Switzerland}
\date{\today}

\begin{document}

\maketitle

\begin{abstract}Let $f$ be a non-negative square-integrable function on a finite volume hyperbolic surface $\Gamma\backslash\Hyp$, and assume that $f$ is non-autocorrelated, that is, perpendicular to its image under the operator of averaging over the circle of a fixed radius $r$. We show that in this case the support of $f$ is small, namely, it satisfies $\mu(\supp{f}) \leq (r+1)e^{-\frac{r}{2}} \mu(\Gamma\backslash\Hyp)$.

As a corollary, we prove a lower bound for the measurable chromatic number of the graph, whose vertices are the points of $\Gamma\backslash\Hyp$, and two points are connected by an edge if there is a geodesic of length $r$ between them. We show that for any finite covolume $\Gamma$ the measurable chromatic number is at least $e^{\frac{r}{2}}(r+1)^{-1}$. 
\end{abstract}

\section{Introduction}
Let $\Hyp$ denote the hyperbolic plane, let $\Gamma\subset \PSL$ be a finite covolume Fuchsian group, and let $r > 0$ be a fixed positive number. Let $A_r:L^2(\Gamma\backslash\Hyp)\to L^2(\Gamma\backslash\Hyp)$ denote the operator of averaging over the circle of radius $r$, in other words, convolution with the uniform measure on the circle of raduis $r$
\[
K \left(\begin{array}{cc}
e^{\frac{r}{2}} & 0\\
0 & e^{-\frac{r}{2}}
\end{array}\right) K,
\]
where $K = PSO(2)\subset \PSL$ is the maximal compact subgroup of $\PSL$. The main result of this paper is the following theorem.
\begin{theorem}\label{thm:main}
Let $f\in L^2(\Gamma\backslash \Hyp)$ such that $f\geq 0$ and $\langle A_r f, f\rangle = 0$. Then 
 $$
\frac{\mu(\supp{f})}{\mu(\Gamma\backslash \Hyp)} \leq \frac{r+1}{e^{\frac{r}{2}}}.
 $$
\end{theorem}
This result can be interpreted as a quantative version of the result of Nevo,~\cite{nevo1994pointwise}, on the equidistribution of the circle on a hyperbolic surface. 
\begin{theorem}[Nevo, \cite{nevo1994pointwise}]\label{thm:nevo}
Let $f\in L^2(\Gamma\backslash \Hyp)$, then 
 \[
 A_r f \to \frac{1}{\mu(\Gamma\backslash \Hyp)}\int_{\Gamma\backslash \Hyp} f d\mu, \text{ as } r\to\infty,
 \]
 almost everywhere and in the $L^2$-norm.
\end{theorem}
Results similar to Theorem~\ref{thm:main} can be deduced from the theorem of Nevo. However, in order to do so, one needs an estimate on the rate of the connvergence, which is governed by the spectral gap of the Laplacian on $\Gamma\backslash \Hyp$. The bound in Theorem~\ref{thm:main} does not exploit any information on the spectral gap, and more importantly, to the best of our knowledge outperforms Theorem~\ref{thm:nevo} in the cases when the spectral gap is not optimal, i.e., there are non-trivial eigenvalues of the Laplacian arising from the complementary series. See Section~\ref{sec:nevo} for more details    

Theorem~\ref{thm:main} implies an interesting corollary for the problem of coloring of hyperbolic surfaces, a generalization of the Hadwiger-Nelson problem. We elaborate on this in the rest of the introduction. The Hadwiger-Nelson problem asks for \textit{the chromatic number of the plane}, that is the minimal number of colors needed in order to color points of the Euclidean plane in such a way that no two points at distance $1$ exactly receive the same color. The problem was originally posed in the  1950's (see~\cite{Soifer2016} for the history of the problem), and in 1961, in the problems section of a journal,~\cite{mosers1961}, the Moser brothers proved a lower bound of 4. They presented a unit-distance graph of chromatic number 4, which was later named Mosers' Spindle. An upper bound of 7 can be proved by considering a 7-coloring of the regular hexagon tessellation of the plane with hexagons of diameter slightly less than $1$. Remarkably, no progress on the original question had been made until 2018, when de Grey,~\cite{Grey2018}, constructed a unit-distance graph of chromatic number 5. His construction was then generalized, see, for example,~\cite{Exoo2019,heule2018computing}, and notably gave rise to a Polymath Project,~\cite{Polymath16}.

One can ask the same question conditioned the color classes to be measurable. The number is then called the \textit{measurable} chromatic number of the plane. In this case the lower bound of $5$ was already proved by Falconer in \cite{falconer1981}. His approach was to give a lower bound on the maximum density of a distance $1$ avoiding set of the plane, that is, a set that contains no two points at distance 1 from each other. The question of the density of distance $1$ avoiding sets in $\R^n$ became interesting on its own, see~\cite{bachoc2015} and references therein. In~\cite{bachoc14}, a spectral approach for providing a lower bound for the Euclidean space of dimension $n$ was presented, as well as it is adaptation to a general setting, of which we make use in this paper. 

An analogous question on a sphere is known as the Witsenhausen problem,~\cite{witsenhausen74}: What is the largest possible surface measure of a subset of $S^{n-1}\subset \R^n$ not containing any pair of points lying at the angle $\pi/2$? A general spectral approach for any $n$ and any forbidden angle $\theta$ (not just $\pi/2$) was developed in~\cite{decorte2015}. Later, the exact convex formulation was given in~\cite{decorte2018complete} extending the spectral approach and providing better bounds.

Note that unlike the Euclidean case, where varying the distance does not change the answer, the chromatic number of the sphere a priori depends on the forbidden angle.

The hyperbolic plane came under consideration relatively recently, when M. Kahle formulated the following question on Math Overflow,~\cite{kahle12}. It also appears as Problem P in~\cite{kloeckner2015coloring}:
\begin{quote}
	Let $\Hyp$ be the hyperbolic plane (with constant curvature $-1$), and
	let $r>0$ be given. If $\Hyp(r)$ is the graph on vertex set $\Hyp$, in which two
	points are joined with an edge when they have distance $r$, then what is
	the chromatic number $\chi(\Hyp(r))$?
\end{quote}

As in the spherical case, the number a priori depends on the forbidden distance. The best known lower bound is $4$, and is given again by the Mosers' spindle, while the best known lower bound for the \emph{measurable} chromatic number is $6$ provided $r$ is large enough (numerical experiments show that $r=12$ is enough) is proved by the spectral method in \cite{DeCorte2019}.

The best known upper bound was proved by Kloeckner in \cite{kloeckner2015coloring} and then improved by Parlier and Petit in \cite{parlier2017chromatic}, and is known to be linear in $r$ for large enough $r$. The bound is 
$$
 \chi (\Hyp(r))\leq 5\left( \left\lceil \frac{r}{\ln{4}}\right\rceil + 1\right), \text{ for } r>5.
$$

As in other settings, a lower bound on the measurable chromatic number $\chi_m(\Hyp(r))$ can be proved via an upper bound on the density of a distance $r$ avoiding set, i.e., a set that have no two points at distance of $r$ from each other. See~\cite{Bowen2002} for a definition of the density of a set on the hyperbolic plane, which is given there in the context of sphere packings, as well as a discussion on its differences from the Euclidean case. 

In this paper, we focus on the \textit{periodic} distance $r$ avoiding sets on the hyperbolic plane. A set on the hyperbolic plane is called periodic, if there exists a Fuchsian group $\Gamma\subset \PSL$ such that the set is invariant under the action of $\Gamma$. A Fuchsian group is a discrete subgroup of $\PSL$. A Fuchsian group $\Gamma$ is called finite covolume, if it has a fundamental region $F\subset\Hyp$ of finite measure (see Subsection~\ref{subsec:Fuchsian} for the definition). In this case, we denote $\mu(\Gamma\backslash\Hyp) = \mu(F)$, and for a $\Gamma$-invariant subset $A\subset\Hyp$, we denote $\mu(\Gamma\backslash A) = \mu(A\cap F)$.  We prove the following theorem.

\begin{theorem}\label{thm:ind-set}
Let $r>0$, and let $\Gamma\subset\PSL$ be a finite covolume Fuchsian group. If a set $I\subseteq \Hyp$ is distance $r$ avoiding and $\Gamma$-invariant, then
$$
\mu(\Gamma\backslash I) \leq (r+1)e^{-\frac{r}{2}} \mu(\Gamma\backslash\Hyp).
$$
\end{theorem}

Every complete finite volume hyperbolic surface $X$ admits a Fuchsian model, that is, it can be presented as a quotient $\Gamma\backslash \Hyp$ for some Fuchsian group $\Gamma\subset \PSL$, see~\cite{katok1992fuchsian} for more details. Let $X(r)$ denote the graph on $X = \Gamma\backslash \Hyp$, where two points $\Gamma z$ and $\Gamma z'$ are connected by an edge iff there exists $\gamma \in \Gamma$ such that $\gamma z$ and $z'$ are at distance $r$ from each other on $\Hyp$. The following result is a direct corollary of Theorem~\ref{thm:ind-set}.

\begin{corollary}\label{cor:main}
 Let $r$ be a positive number, $X$ be a complete hyperbolic surface of finite volume. Then 
 $$
 \chi_m(X(r)) \geq e^{\frac{r}{2}}/{(r+1)}.
 $$
\end{corollary}

In~\cite{Par2016}, Parlier and Petit proved an upper bound on the chromatic number of $X(r)$ (exponential in $r$ and independent of $X$) and constructed a family of surfaces for which an exponential lower bound holds. 
\begin{theorem}[Parlier and Petit,~\cite{Par2016}]\label{thm:PP-upper-bound}
There exists a constant $C_1 > 0$ such that for every number $r > 0$ every complete hyperbolic surface $X$ satisfies
$$
 \chi(X(r)) \leq C_1 e^r.
$$
There exists a constant $C_2 > 0$  and a family of complete hyperbolic surfaces $X_r$, $r > 0$, so that
$$
 \chi(X_r(r)) \geq C_2 e^{r/2}.
$$
\end{theorem}

Note that the chromatic number $\chi(X(r))$ of $X(r)$ is always less than or equal to the measurable one $\chi_m(X(r))$. The upper bound for $\chi(X(r))$ is achieved by construcing periodic colorings on $\Hyp$. It seems like the proof~\cite{Par2016} can be adjusted to give a bound for the measurable chromatic as well, but as it is given in the paper, it requires a random choice of color for points in possibly a measurable subset, which may lead to non-measurable color class.

The circumference of the circle of radius $r$ on the hyperbolic plane is $2\pi \sinh{r}$ can be thought as the degree of a vertex in the graph $X(r)$. Then the upper bound in Theorem~\ref{thm:PP-upper-bound} can be interpreted as a generalization of the bound of $k+1$ on the chromatic number of a $k$-regular graph, proved by~\cite{brooks1941}. Continuing this analogy, Corollary~\ref{cor:main} would read as lower bound of order $\sqrt{k}/\log{k}$ on the chromatic number of a $k$-regular graph. While for finite graphs this is far from being true (as for instance, there are bipartite graphs), the corollary above shows that it holds for graphs $X(r)$ for all hyperbolic surfaces of finite volume.

\paragraph*{Structure of the Paper.} Section~\ref{sec:prel} contains the needed preliminaries. Section~\ref{sec:proof} contains a proof of Theorem~\ref{thm:main}. We conclude with Section~\ref{sec:nevo}, where we discuss relation of Theorem~\ref{thm:main} to the known results.

\section{Preliminaries}\label{sec:prel}

In this section, we provide the required preliminaries. The material here largely overlaps with the corresponding sections in \cite{DeCorte2019} and \cite{Golubev2019} coauthored with Evan DeCorte and Amitay Kamber, respectively.

\subsection{The Hoffman bound for the measurable independence ratio of an infinite graph}\label{sec:hoffmangeneral}

The Hoffman bound, proved in~\cite{hoffman1970eigenvalues}, is a well-known spectral upper bound for the independence ratio of a finite graph. Hoffman proved a bound for the chromatic number, but later a different proof, via upper-bounding the independence ratio, appeared. Recall that the chromatic number $\chi(G)$ of a graph $G$ is the minimum number of colors needed in order to color the vertices in such a way that no edge is monochromatic, while the independence ratio $\alpha(G)$ is the cardinality of the largest subset of vertices that does not contain an edge normalized by the cardinality of the vertex set. In particular, these numbers satisfy $\alpha(G)\chi(G) \geq 1$. In this section, we provide the reader the necessary background for understanding the Hoffman bound applied to infinite graphs detailed in \cite{bachoc14}. We begin with the usual statement of Hoffman's theorem for the context of a finite regular graph.

\begin{theorem}[Hoffman,~\cite{hoffman1970eigenvalues}]\label{thm:hoffman-finite} 
	Let $G$ be a finite regular graph with at least one edge, and let $A$ be its adjacency matrix. Denote the maximum and, resp., minimum eigenvalues of $A$ by $M$ and $m$. Then
	\[
		\alpha(G) \leq \frac{-m}{M-m},
	\]
	and since $\alpha(G)\chi(G) \geq 1$,
	\[
	  \chi(G) \geq \frac{M-m}{-m}.
	\]

\end{theorem}

Note that the $m$ in Theorem~\ref{thm:hoffman-finite} will always be negative, since the trace of the adjacency matrix is $0$. The proof of Theorem \ref{thm:hoffman-finite} can be generalized to infinite graphs, provided correct analogues of notions such as \emph{adjacency matrix}, \emph{independence} and \emph{coloring} can be found. We now give those definitions and state the infinite version of Hoffman's bound, following \cite{bachoc14}.

Now let $(X, \Sigma, \mu)$ be a measure space. Write $L^2(X)$ for $L^2(X, \mu)$, we let $A~:~L^2(X)\to~L^2(X)$ be a bounded, self-adjoint operator. We say that a measurable subset
$I \subseteq X$ is \emph{$A$-independent} if
\begin{align}\label{eq:ind}
	\langle Af, f \rangle = 0
\end{align}
whenever $f \in L^2(X)$ is a function which vanish almost everywhere outside $I$. Assume now that $\mu(X) < \infty$, then the independence ratio of the operator $A$ is
\[
 \alpha(A) = \sup\left\{\frac{\mu(I)}{\mu(X)}\mid I \text{ is independent}\right\}.
\]

When $I$ is an independent vertex subset of a finite graph $G$, and $A$ is the adjacency matrix of $G$, then one easily verifies \eqref{eq:ind} for all $f$ supported on $I$. One may therefore regard the operator $A$ as playing the role of an adjacency operator. 

An \emph{$A$-measurable coloring} is a partition $X = \dot{\bigcup}_i C_i$ of $X$ into $A$-independent sets. In this case the sets $C_i$ are called \emph{color classes}. The \emph{$A$-chromatic number of $X$}, denoted $\chi_A(X)$, is the smallest number $k$ (possibly $\infty$), such that there exists an $A$-measurable coloring using only $k$ color classes.

Recall that operators on infinite-dimensional Hilbert spaces need not have eigenvectors, even when they are self-adjoint, so Theorem~\ref{thm:hoffman-finite} does not immediately extend to the infinite case. It turns out that for our purposes, the correct analogues of the $M$ and $m$ of Theorem~\ref{thm:hoffman-finite} are given by the following definitions.
\begin{align*}
	M(A) = \sup_{\|f\|_2 = 1} \langle Af, f \rangle,\\
	m(A) = \inf_{\|f\|_2=1} \langle Af, f \rangle.
\end{align*}

The numbers $\langle Af, f \rangle$ are real since $A$ is self-adjoint, and $M(A)$ and $m(A)$ are finite since $A$ is bounded. In \cite{bachoc14}, the following extension of Hoffman's bound is proven.

\begin{theorem}[Bachoc, DeCorte, de Oliveira and Vallentin,~\cite{bachoc14}]\label{thm:hoffman}
	Let $(X, \Sigma, \mu)$ be a probability space and suppose $A~:~L^2(X)\to~L^2(X)$ is a nonzero, bounded, self-adjoint operator. Denote by $1_X\in L^2(X)$ the all-one function on $X$, fix $R\in \R$, and let $\e = \|A1_X - R\cdot1_X\|$. Then
	\[
		\chi_A(X) \geq \frac{M(A)-m(A)}{-m(A)}.
	\]
	and, if $R-m(A)-\e > 0$, we have
	\[
	 \alpha(A) \leq \frac{-m(A)+2\e}{R-m(A)-\e}.
	\]
\end{theorem}

It should be noted that unlike in Theorem~\ref{thm:hoffman-finite} where
$A$ is simply the adjacency matrix, with infinite graphs there may be no
canonical choice for $A$.
The choice of $A$ determines which sets are to be considered as ``independent'',
and thus admissible as color classes. When applying Theorem~\ref{thm:hoffman}
for an infinite graph $G$ on a measurable vertex set,
one therefore typically chooses $A$ to define a class of independent sets which is larger that the class of true measurable independent sets of $G$, for then the measurable chromatic
number of $G$ is at least $\chi_A(X)$. 

Also note that Theorem~\ref{thm:hoffman} applies equally well when $X$ is finite. Theorem~\ref{thm:hoffman} then says that to obtain a lower bound on the chromatic number of a finite graph, one may optimize over symmetric matrices satisfying \eqref{eq:ind}. For finite graphs this was developed in~\cite{lovasz79} and gave rise to the notion of the Lov\'asz $\theta$-number.

\subsection{The hyperbolic plane}

There are several models for the hyperbolic plane $\Hyp$ of constant curvature $-1$, and we stick to the upper half-plane model. That is, the complex half-plane $\Hyp = \{z=x+iy\in\mathbb{C}\mid\text{Im}(z)>0\}$ endowed with the metric $(ds)^2=\frac{(dx)^2+(dy)^2}{y^{2}}$. The distance $d(z,z')$ between $z=x+iy,\:z'=x'+iy'\in\mathbb{H}$ in this model can be calculated as 
\begin{equation*}
d\left(z,z'\right)=\text{acosh\ensuremath{\left(1+\frac{\left(x'-x\right)^{2}+\left(y'-y\right)^{2}}{2yy'}\right).}}
\end{equation*}
The group $G=PSL_{2}(\R)$ acts on $\mathbb{H}$ by M\"obius transformations,
i.e., 
\begin{equation*}
\left(\begin{array}{cc}
a & b\\
c & d
\end{array}\right)\cdot z=\frac{az+b}{cz+d},
\end{equation*}
and constitutes the group of orientation preserving isometries of $\mathbb{H}$. In this paper, we identify an element $g\in G=PSL_{2}(\R)$ with its preimage in $SL_{2}(\R)$, i.e., omit the $\pm$ sign. The group $G$ acts transitively on the points of $\mathbb{H}$, with the subgroup $K=PSO_{2}(\mathbb{R})\subset G$ being the stabilizer of the point $i$, to which we refer as the origin of $\Hyp$. The subgroup $K$ acts on $\Hyp$ by rotations around $i$. The plane $\mathbb{H}$ can be identified with the quotient $G/K$, and in particular, the circle of radius $r$ around $i$ identifies with the double coset 
\[                                                                                                                                                                                                                                                                                                                                                                                                                                                                                                                                                                                                          
K\left(\begin{array}{cc}
e^{r/2} & 0\\
0 & e^{-r/2}
\end{array}\right)K.                                                                                                                                                                                                                                                                                                                                                                                                                                                                                                                                                                                                        \]
The Haar measure on $G$ which is normalized so that the measure of $K$ is equal to $1$ agrees with the standard measure $\mu$ on $\Hyp$.

\subsection{Harmonic analysis on $\mathbb{H}$}

For $f\in L^{1}\left(\Hyp\right)$, its Helgason-Fourier transform $\widehat{f}(s,k)\in C\left(\C\times K\right)$, is defined as 
\begin{equation*}
\widehat{f}(s,k)=\int_{\mathbb{H}}f(z)\overline{\left(\text{Im}(kz)\right)^{\frac{1}{2}+is}}dz,
\end{equation*}
for $s\in\mathbb{C}$ and $k\in K=PSO_{2}(\mathbb{R})$ whenever the integral exists.

In the case when $f$ is $K$-invariant, i.e., $f(kz)=f(z)$ for all $z\in\Hyp$ and $k\in K$, its transform is independent of $k$ and can be written with the help of the spherical functions. For every
$s\in\C$, the corresponding spherical function is a $K$-invariant function on $\Hyp$ defined as 
\begin{equation*}
\varphi_{\frac{1}{2}+is}(z)=\int_{K}\overline{\left(\text{Im}\left(kz\right)\right)^{\frac{1}{2}+is-2}}dk.
\end{equation*}
Since $\varphi_{\frac{1}{2}+is}$ is $K$-invariant, it depends solely on the hyperbolic distance from a point to the origin $i$, and can be written as
\begin{equation*}
\varphi_{\frac{1}{2}+is}(z)=\varphi_{\frac{1}{2}+is}(ke^{-r}i)=P_{-\frac{1}{2}+is}(\cosh r),
\end{equation*}
where $k\in K$, $r\in\R_{\ge0}$ is the distance from $z$ to $i$, and $P_{s}(r)$ is the Legendre function of the first kind. We also denote $\phi(s,r)=\varphi_{\frac{1}{2}+is}(e^{-r}i)$, and note that
for $s\in\mathbb{R}$ (see~\cite[Lemma 7]{DeCorte2019} or \cite[Exercise 3.2.28]{terras2013harmonic})
\begin{equation*}
\phi(s,r)=\frac{\sqrt{2}}{\pi}r\int_{0}^{1}\frac{\cos\left(srx\right)}{\sqrt{\cosh r-\cosh rx}}dx.
\end{equation*}

The Helgason-Fourier transform of a $K$-invariant function $f$ reads as
\begin{equation*}
\widehat{f}(s)=\int_{\mathbb{H}}f(z)\varphi_{\frac{1}{2}+is}(z)dz=\int_{0}^{\infty}f(e^{-r}i)P_{-\frac{1}{2}+is}(\cosh r)\sinh rdr.
\end{equation*}

For two functions $f_{1},f_{2}\in L^{1}\left(\Hyp\right)$, their convolution is defined as
\begin{equation*}
f_{1}\ast f_{2}(z)=\int_{G}f_{1}(gi)f_{2}(g^{-1}z)dg.
\end{equation*}
We exploit the following properties of the Helgason-Fourier transform on $\Hyp$. For an extensive presentation of the theory, see \cite{helgason1984groups,terras2013harmonic}.
\begin{prop}
\label{prop:Plancherel-Helgason}\cite[Theorem 3.2.3]{terras2013harmonic}
\begin{enumerate}
\item (Plancherel Formula) The map $f\to\widehat{f}$ extends to an isometry between $L^{2}\left(\mathbb{H},d\mu\right)$ and $L^{2}\left(\mathbb{R}\times K,\frac{1}{4\pi}s\tanh\pi s\,dsdk\right)$,
where $K$ is identified with $\R/\Z$.
\item (Convolution property) For $f,g\in L^{1}(\mathbb{H})$, where $g$
is $K$-invariant, 
\begin{equation*}
\widehat{f\ast g}=\widehat{f}\cdot\widehat{g},
\end{equation*}
where $\ast$ stands for convolution, and $\cdot$ for pointwise multiplication.
\end{enumerate}
\end{prop}

The Helgason-Fourier transform can be extended to compactly supported measures on $\Hyp$. Namely, for such a measure $\nu$, its transform $\widehat{\nu}(s,k)\in C\left(\C\times K\right)$, is defined for $s\in\mathbb{C}$ and $k\in K=PSO_{2}(\mathbb{R})$ as 
\begin{equation*}
\widehat{\nu}(s,k)=\int_{\mathbb{H}}\overline{\left(\text{Im}(k(z))\right)^{\frac{1}{2}+is}}d\nu,
\end{equation*}
and, if the measure is $K$-invariant, its transform is independent of $k$, and can be written as
\begin{equation*}
\widehat{\nu}(s)=\int_{\Hyp}\varphi_{\frac{1}{2}+is}(z)d\nu.
\end{equation*}

We will need the following claim, which follows from Proposition~\ref{prop:Plancherel-Helgason}. 
\begin{corollary}
\label{cor:L2 fourier transform}Let $\nu$ be a compactly supported measure on $\Hyp$, and assume that $\widehat{\nu}\in L^{2}\left(\mathbb{R}\times K,\frac{1}{4\pi}s\tanh\pi s\,dtdk\right)$.
Then $\nu$ can be represented as an $L^{2}$-function on $\Hyp$, i.e.,
there exists $f_{\nu}\in L^{2}\left(\Hyp\right)$ such that for every
$f\in C_{c}\left(\Hyp\right)$, $\nu\left(f\right)=\int f_{\nu}(z)f(z)dz$.
\end{corollary}

\subsection{Adjacency on the hyperbolic plane}

For $r>0$, let $\Hyp(r)$ be the graph whose vertex set is the hyperbolic plane $\Hyp$, where two points are joined with an edge precisely when their distance is equal to $r$. For any $f \in L^2(\Hyp)$, $r>0$, and $z \in \Hyp$, we define $(A_r f )(z)$ to be the average of $f$ around the hyperbolic circle of radius $r$ centered at $z$. It is not hard to check that $A_r$ is self-adjoint, and bounded with norm $1$. The operator $A_r$ can be thought of as an adjacency operator for the graph $\Hyp(r)$, in the sense that any measurable independent set of $\Hyp(r)$ is also $A_r$-independent.

We briefly comment here on our choice of the operator $A_r$. Let $B$ be any operator which includes the independent sets of $\Hyp(r)$ in its collection of $B$-independent sets. Let $k \in K$ be any rotation of $\Hyp$ fixing the origin and let $R_k : L^2(\Hyp) \to L^2(\Hyp)$ denote precomposition with $k^{-1}$:
\[
	R_k f := f \circ k^{-1}.
\]
Let $B' = R_k^{-1} B R_k$. Then the independent sets for $\Hyp(r)$ will also be $B'$-independent, and moreover, it is clear that $M(B) = M(B')$ and $m(B) = m(B')$. Therefore, when trying to choose which operator to use in Theorem~\ref{thm:hoffman}, we may as well replace the operator $B$ with the ``symmetrized'' operator 
\[
	S := \int_{k \in K} R_k^{-1} B R_k~dk,
\]
which is rotationally invariant in the sense that $R_k^{-1} S R_k = S_d$ for all $k \in K$. In other words, without loss of generality, we may take the operator $B$ to be rotation-invariant. Our choice of $A_r$ was therefore natural.

\subsection{The numerical range of $A_{r}$ on $\Hyp$}

For $r>0$, let $A_{r}$ denote the operator on $C(\mathbb{H})$ that averages a function over a circle of radius $r$, i.e., for a function $f\in C(\mathbb{H})$ and $z=gi\in\mathbb{H}$ ($g\in G$),
\begin{equation*}
\left(A_{r}f\right)(z)=\intop_{K}f\left(gk\left(\begin{array}{cc}
e^{r/2} & 0\\
0 & e^{-r/2}
\end{array}\right)i\right)dk.
\end{equation*}
The operator $A_{r}$ is bounded and self-adjoint with respect to the $L^{2}$-norm on $L^{2}\left(\Hyp\right)\cap C\left(\Hyp\right)$, so it extends to a self-adjoint operator $A_{r}\colon L^{2}\left(\Hyp\right)\to L^{2}\left(\Hyp\right)$. By duality, we may also extend $A_{r}$ to an operator on the compactly supported measures on $\Hyp$. Note that the operator $A_{r}$ can be written as a convolution from the right with a uniform $K$-invariant probability measure $\delta_{S_{r}}$ supported on the double coset $K\left(\begin{array}{cc}
e^{r/2} & 0\\
0 & e^{-r/2}
\end{array}\right)K$, i.e.,
\begin{equation*}
A_rf=f \ast \delta_{S_{r}}.
\end{equation*}

The spherical functions $\varphi_{\frac{1}{2}+is}$ on $\Hyp$ are eigenfunctions of $A_{r}$ for every $r>0$, namely, 
\[
A_{r}\varphi_{\frac{1}{2}+is} =P_{-\frac{1}{2}+is}(\cosh r)\cdot\varphi_{\frac{1}{2}+is}.
\]
In particular, the following lemma follows from Proposition~\ref{prop:Plancherel-Helgason} and Corollary~\ref{cor:L2 fourier transform}:
\begin{lemma}
The $L^{2}$-spectrum of $A_{r}$ on $\Hyp$ is the set $W_0(A_r) = \left\{ P_{-\frac{1}{2}+is}(\cosh r)\mid s\in\R\right\} $, and hence, 
\[
 M(A_r) = \sup W_0(A_r),\text{ and } m(A_r) = \inf W_0(A_r).
\]
\end{lemma}

\subsection{Fuchsian groups}\label{subsec:Fuchsian}

A discrete subgroup $\Gamma\subseteq \PSL$ is called a \emph{Fuchsian group}. Every hyperbolic surface admits a Fuchsian model, i.e., can be represented as a quotient of $\Hyp$ by a Fuchsian group. 

For a Fuchsian group $\Gamma \subseteq \PSL$, a closed region $F\subseteq \Hyp$ (that is, a closure of an open non-empty set $F_0\subseteq \Hyp$, called the interior of $F$) is called a \emph{fundamental region} for $\Gamma$ if 
\[
 \bigcup_{\gamma\in\Gamma} \gamma F = \Hyp \text{ and } F_0\cap \gamma F_0 = \emptyset,\, \forall 1\neq \gamma\in \Gamma.
\]
Since two different fundamental regions for $\Gamma$ have the same measure (\cite[Theorem 3.1.1]{katok1992fuchsian}), we refer to it as the measure of the quotient $\mu(\Gamma\backslash\Hyp)$. If it is finite, we call $\Gamma$ finite covolume. Note that in this case, the all-one function belongs to $L^2(\Gamma\backslash \Hyp)$. 

\subsection{The operator $A_r$ and its numerical range on the quotients}

Let $\Gamma\subseteq \PSL$ to be a Fuchsian group. A measurable subset $I\subseteq \Hyp$ is called $\Gamma$-invariant iff for every $1\neq \gamma\in\Gamma$
\[
 \mu\left(I\triangle \gamma I\right) = 0.
\]

We denote by $\mu(\Gamma\backslash I)$ the measure $\mu (I\cap F)$, where $F$ is a fundamental region for $\Gamma$. Since $I$ is $\Gamma$-invariant, it is independent of $F$.

An equivalent way to study $\Gamma$-invariant independent sets on the graph $\Hyp(r)$, is to consider the following graph. Its vertex set is $\Gamma\backslash \Hyp$, and two points $\Gamma z$ and $\Gamma z'$ are connected by an edge iff there exists $\gamma\in\Gamma$ such that $\gamma z$ and $z'$ are at distance $r$ from each other on $\Hyp$. 

In order to apply the machinery of Theorem~\ref{thm:hoffman}, we consider the actions of $A_{r}$ on $L^{2}\left(\Gamma\backslash\mathbb{H}\right)$. In this case the spectrum is not necessarily discrete, but one may still associate to every point of the spectrum a spherical function $\varphi_{\frac{1}{2}+is}$. The value $\frac{1}{2}+is\in\mathbb{C}$ is called a "unitary dual parameter" and the union of all the unitary dual parameters across the spectrum is call the unitary dual of $X=\Gamma\backslash\Hyp$. Namely, if $\frac{1}{2}+is\in\mathbb{C}$ appears in the unitary dual of $\Gamma\backslash\Hyp$, then $P_{-\frac{1}{2}+is}(\cosh r)$ is in the spectrum of $A_{r}$ (and $\frac{1}{4}+s^{2}$ is an eigenvalue of the Laplacian $\Delta$). It is well-known that, in general, the unitary dual of $\Gamma\backslash\Hyp$ is contained in the set $\left\{ \frac{1}{2}+is\mid s\in\R\right\} \cup\left\{ \frac{1}{2}+is\mid is\in\left(-\frac{1}{2},\frac{1}{2}\right)\right\} \cup\left\{ 0,1\right\}$ (see e.g., \cite[Section 5.2]{lubotzky1994discrete}). The set $\left\{ \frac{1}{2}+is\mid s\in\R\right\}$ is called the principal series, the set $\left\{ \frac{1}{2}+is\mid is\in\left(-\frac{1}{2},\frac{1}{2}\right)\right\}$ is called the complementary series, and $\{0,1\}$ is called trivial. The trivial part corresponds to the constant function on $\Gamma\backslash \Hyp$. 

\begin{lemma}\label{lem:num-range}
For a Fuchsian group $\Gamma\subseteq \PSL$, the $L^{2}$-spectrum of $A_{r}$ on $\Gamma \backslash \Hyp$ is contained in the set $W(A_r) = \left\{ P_{-\frac{1}{2}+is}(\cosh r)\mid s\in\R\text{ or } is\in \left[-\frac{1}{2},\frac{1}{2}\right]\right\}$. In particular, the numerical range of $A_r$ on $L^2(\Gamma\backslash \Hyp)$ satisfies
\[
 [m(A_r),M(A_r)] \subseteq [\inf W(A_r), \sup W(A_r)].
\]
\end{lemma}

\section{Proofs of Theorems~\ref{thm:main} and~\ref{thm:ind-set}}\label{sec:proof}

The proof of Theorem~\ref{thm:main} goes in two steps. We first prove Theorem~\ref{thm:ind-set} on the size of a periodic distance $r$ avoiding set. Then we show that if a function satisfies the condition of Theorem~\ref{thm:main}, then its support satisfies the condition of Theorem~\ref{thm:ind-set}.

\begin{proof}[Proof of Theorem~\ref{thm:ind-set}]

Operator $A_r$ acts on $L^2(\Gamma\backslash\Hyp)$ as the averaging operator over a sphere of radius $r$ around a point, hence its operator norm is bounded by $1$. Since $\Gamma$ is of finite covolume, the all-one on $\Gamma\backslash\Hyp$ function belongs to $L^2(\Gamma\backslash\Hyp)$, and it is an eigenfunction of $A_r$ with eigenvalue $1$. This implies that $M(r)=1$. Note that the all-one function does not belong to $L^2(\Hyp)$, and this makes the analysis of periodic and non-periodic colorings of $\Hyp$ (as in~\cite{DeCorte2019}) essentially different.

In order to show a lower bound on $m(r)$, we show a lower bound on the whole set 
\[
W(A_r) = \left\{ P_{-\frac{1}{2}+is}(\cosh r)\mid s\in\R\text{ or } is\in \left[-\frac{1}{2},\frac{1}{2}\right]\right\}.
\]
The following expression holds for $P_{-\frac{1}{2} + i\cdot s}(\cosh{r})$ (see \cite[Lemma 3.2]{DeCorte2019}).
$$
P_{-\frac{1}{2} + i\cdot s}(\cosh{r}) = \frac{1}{\sqrt{2}\cdot \pi} \int_{-r}^{r}  \frac{e^{i\cdot s\cdot x}}{\sqrt{\cosh(r)-\cosh(x)}} dx.
$$

Note that for $is\in [-\frac{1}{2},\frac{1}{2}]$, the expression under the integral is positive, and hence $P_{-\frac{1}{2} + i\cdot s}(\cosh{r})$ is also positive. For $s\in \mathbb{R}$, the following bound was shown in \cite[Proposition 2.3]{Golubev2019} that
$$
\left|P_{-\frac{1}{2} + i\cdot s}(\cosh{r})\right|\leq (r+1)e^{-\frac{r}{2}}.
$$

To summarize, for a finite covolume Fuchsian group $\Gamma \subseteq \PSL$, we have
\[
 M(A_r) = 1,\text{ and } m(A_r) \geq - (r+1) e^{-\frac{r}{2}}.
\]

In order to apply Theorem~\ref{thm:hoffman} it is left to note that the all-one function $1_{\Gamma\backslash\Hyp}$ is an eigenfunction of $A_r$ on $\Gamma\backslash\Hyp$, and hence we can take $R = M(A_r) = 1$ and $\e = 0$, and the bound reads as
\[
 \frac{\mu(\Gamma\backslash I)}{\mu(\Gamma\backslash\Hyp)}  \leq \frac{(r+1)e^{-\frac{r}{2}}}{1+(r+1)e^{-\frac{r}{2}}} \leq  (r+1)e^{-\frac{r}{2}} .
\]

\end{proof}

\begin{proof}[Proof of Theorem~\ref{thm:main}.]
In order to prove the theorem it is enough to show that $\supp{f}$ is an $A_r$-independent up to a subset of measure zero. In order to do so, we show that up to a subset of measure zero $\supp{f}$ contains no two points such that there exists a geodesic of length $r$ in $X=\Gamma\backslash \Hyp$ between them. 

It follows from the finiteness of the measure on $X$, via H\"{o}lder's inequality, that $L^2(X)\subset L^1(X)$. Hence, $f$ is an integrable function as it is also non-negative. An analogue of the Lebesgue Differentiation Theorem holds for $X$ (see~\cite[Theorem 2.9.8]{federer1969grundlehren}), and it implies that almost all points of $\supp{f}$ are density points, that is, for almost all $x\in\supp{f}$, the following holds
\[
 \lim_{\delta\to0^+}\frac{\mu\left(\supp{f}\cap B(x,\delta)\right)}{\mu\left(B(x,\delta)\right)} = 1,
\]
where $B(x,\delta)$ is the open ball of radius $\delta$ around $x$. Now assume that there exists two density points $x_1,x_2\in X$ such that there exists a geodesic of length $r$ in $X$ between them. There exists $\delta > 0$ such that 
\begin{equation}\label{eq:inter}
 \frac{\mu\left(\supp{f}\cap B(x_i,\delta)\right)}{\mu\left(B(x_i,\delta)\right)} \geq \frac{2}{3},\text{ for }i=1,2.
\end{equation}
Denote $B_i = \supp{f}\cap B(x_i,\delta)$ for $i=1,2$. Then since there exists a geodesic of length $r$ between $x_1$ and $x_2$, 
\[
 \langle A_r \chi_{B_1}, \chi_{B_2} \rangle > 0,
\]
where $\chi_{B_i}$ is the characteristic function of $B_i$ for $i=1,2$. This follows from Inequality~\ref{eq:inter} and the fact that the circle of radius $r$ centered at an inner point $y\in B(x_2,\delta)$ intersects $B(x_1,\delta)$ by an arc of positive length, see Figure~\ref{fig:inter}.

\begin{figure}[h!]\centering
  \begin{tikzpicture}[scale=1]
\node[label={[shift={(0,-0.1)}]$x_1$}, circle,draw=black, fill=black, inner sep=0pt,minimum size=3] (x1) at (0,0) {};
\draw  (x1) -- node[left] {$\delta$} ++ (-80:0.7);

\node[label={[shift={(0,-0.1)}]$x_2$}, circle,draw=black, fill=black, inner sep=0pt,minimum size=3] (x2) at (4,0) {};
\draw  (x2) -- node[right] {$\delta$} ++ (-80:0.7);

\draw  (x1) -- node[above] {$r$} ++ (x2);

\node[circle,draw=black, inner sep=0pt,minimum size=40] () at (0,0) {};

\node[circle,draw=black, inner sep=0pt,minimum size=40] () at (4,0) {};

\node[label={[shift={(0.2,-0.1)}]$y$}, circle,draw=black, fill=black, inner sep=0pt,minimum size=3] (y) at (4+0.5,0.3) {};
\draw [black,thick,domain=-30:30] plot ({4+0.5-4*cos(\x)}, {0.3-4*sin(\x)});
\draw  (y) -- node[right] {$r$} ++ (160:4);
\end{tikzpicture}
\caption{The circle of radius $r$ centered at $y\in B(x_2,\delta)$ intersects $B(x_1,\delta)$ by an arc of positive length.}\label{fig:inter}
\end{figure}
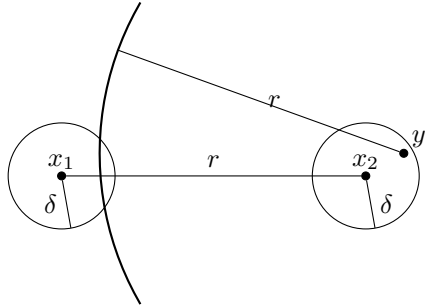
Since $f$ is positive on both $B_1$ and $B_2$ we get that 
\[
\langle A_r f, f \rangle = \int_{X} (A_r f)\cdot f d\mu \geq \int_{B_1} (A_r f)\cdot f d\mu > 0.
\]
\end{proof}

\section{Relation to Known Results}\label{sec:nevo}
The relation between the geometry of a hyperbolic surface and the spectra of various operators on it has been extensively studied for more than thirty years. The averaging operators $A_r$ are not exceptions. One interesting example is the following result on the equidistribution of the circle of radius $r$ in $\Hyp$ on $\Gamma\backslash \Hyp$, which  was originally proved by Nevo,~\cite{nevo1994pointwise}, has become a classical result in ergodic theory, see e.g.~\cite[Chapter 7]{anantharaman2010theoremes}. 
\begin{theorem}[Nevo, \cite{nevo1994pointwise}]\label{thm:nevo}
Let $f\in L^2(\Gamma\backslash \Hyp)$, then 
 \[
 A_r f \to \frac{1}{\mu(\Gamma\backslash \Hyp)}\int_{\Gamma\backslash \Hyp} f d\mu, \text{ as } r\to\infty,
 \]
 almost everywhere and in the $L^2$ norm.
\end{theorem}
The speed of converges is governed by the largest in absolute value non-trivial eigenvalue of $A_r$, which can be expressed via the spectral gap of the Laplacian on $\Gamma\backslash \Hyp$. Similarly, the mixing rate of the geodesic and horocycle flows is bounded in the terms of the spectral gap in the work of Ratner,~\cite{ratner1987rate}. In both cases, bounds get worse as the spectral gap gets smaller. In particular, the best bounds are achieved when there is no non-trivial spectrum arising from the complementary series. 

One can deduce results similar in nature to the main results of this paper, Theorems~\ref{thm:main} and~\ref{thm:ind-set}, from Theorem~\ref{thm:nevo} in the same way the Expander Mixing Lemma,~\cite[Lemma 2.5]{hoory2006expander}, provides a bound on the size of an independent set in a finite graph (see e.g. the proof of Theorem~6.1 in \cite{evra2015mixing} exploiting this idea in a more general setting). More precisely, let $r > 0$ be fixed, let $X = \Gamma\backslash \Hyp$ be a fintie volume hyperbolic surface and let $I\subset X$ be a distance $r$ avoiding set on it. Then one can deduce that
\[
 \frac{\mu(I)}{\mu(X)} \leq \frac{\beta}{1+\beta},
\]
where $\beta$ is the operator norm of $A_r$ acting on $L_0^2(X) = \left\{g\in L^2(X)\mid \langle g,\chi_X\rangle = 0 \right\}$. Let $\lambda > 0$ be the smallest non-trivial eigenvalue of the Laplacian on $X$. If there are no non-trivial spectrum of $A_r$ arising from the complementary series, then $\lambda$ is at least $\frac{1}{4}$, and $\beta$ can be expresed as (see~\cite{ratner1987rate})
\[
 \beta = \min\left\{\frac{r}{2}e^{-\frac{r}{2}},\left(1+|1+4\lambda|^{-\frac{1}{2}}\right)e^{-\frac{r}{2}}\right\}.
\]
In particular, it can outperform the bound in Theorem~\ref{thm:main}, if  $\lambda$ is large.
If $\lambda < \frac{1}{4}$, then 
\[
 \beta = \min\left\{\frac{r}{2}e^{-C\frac{r}{2}},\left(|1+4\lambda|^{-\frac{1}{2}}\right)e^{-C\frac{r}{2}}\right\}
\]
where $0\leq C <1$, and hence, the above bound performs worse than Theorem~\ref{thm:main}.

\section*{Acknowledgements}
The author would like to thank Fernando Oliveira and Alexander Gorodnik for fruitful discussions and suggestions. The author is grateful to Frank Vallentin for valuable comments and remarks on the first version of the manuscript. 
\subsection*{Funding}
The author is supported by the SNF grant number 20002\_169106.

\bibliographystyle{plain}
\bibliography{mybib}

\end{document}